\documentclass[letterpaper,english, 12pt]{article}
\usepackage{babel} 
\usepackage[ansinew]{inputenc}
\usepackage{amsmath, amssymb, amsthm}
\usepackage{nicefrac}
\usepackage[arrow, matrix, curve]{xy}
\usepackage{cite}
\usepackage{graphics}
\usepackage{graphicx}
\usepackage[lmargin=3.00cm,rmargin=3.00cm,tmargin=3cm,bmargin=3cm]{geometry}

\newcommand{\R}{\mathbb{R}}

\newcommand{\Z}{\mathbb{Z}}
\newcommand{\Q}{\mathbb{Q}}
\newcommand{\C}{\mathbb{C}}
\newcommand{\curlyC}{\mathcal{C}}

\newcommand{\F}{\mathcal{F}}

\renewcommand{\H}{\mathcal{H}}

\newcommand{\colim}{\textnormal{colim}}

\newcommand{\into}{\hookrightarrow}

\newcommand{\M}{\mathcal{M}}

\newcommand{\hcob}{\textit{h}-cobordism }
\newcommand{\hcobs}{\textit{h}-cobordisms }

\newcommand{\rk}{\textnormal{rk}}
\newcommand{\z}{\zeta}

\renewcommand{\L}{\mathcal{L}}

\renewcommand{\Re}{\textnormal{Re}}
\renewcommand{\Im}{\textnormal{Im}}
\renewcommand{\tilde}{\widetilde}
\renewcommand{\bar}{\overline}

\newtheorem{Lemma}{Lemma}[section]
\newtheorem{Theorem}[Lemma]{Theorem}
\newtheorem*{Theorem*}{Theorem}
\newtheorem{Proposition}[Lemma]{Proposition}

\theoremstyle{definition}

\theoremstyle{remark}
\newtheorem{Remark}[Lemma]{Remark}

\begin{document}

\title{An Equivariant Version of Hatcher's $G/O$ Construction}
\author{Thomas Goodwillie, Kiyoshi Igusa\thanks{The second author is supported by NSA Grant \#H98230-13-1-0247.}, Christopher Ohrt}

\maketitle

 \begin{abstract}
Let $\H^s(BC_n)$ be the space of stable \hcobs of the classifying space of a cyclic group of order $n$. We explicitly construct generators of the rational homotopy groups $\pi_\ast\H^s(BC_n)\otimes \Q$ by generalizing a construction of Hatcher. This result will be used in a separate paper by the third author to classify axiomatic higher twisted torsion invariants.
\end{abstract} 

\tableofcontents

\section*{Introduction}

For a finite complex $X$ let $\H^s(X)$ be the space of stable \hcobs of $X$. This infinite loopspace is a delooping of the stable concordance space of (any smooth compact manifold that is homotopy equivalent to) $X$. The definition is extended to infinite complexes by direct limit over finite subcomplexes. The group $\pi_0\H^s(X)$ is the Whitehead group of $\pi_1(X)$. When $X$ is the classifying space of a finite group $\Gamma$, the homotopy groups of $\H^s(X)$ are finitely generated and coincide rationally with the algebraic $K$-groups of the group ring:
$$
\pi_j\H^s(B\Gamma)\otimes \Q \cong K_{j+1}(\Z\lbrack \Gamma\rbrack)\otimes \Q .
$$
 
In the case when $\Gamma$ is trivial, $\pi_j\H^s(*)$ is the direct limit over $M$ of $\pi_j\H(D^M)$, where $\H(D^M)$ is the space of \hcobs of the $M$-disk. By \cite{0691.57011} it is independent of $M$ when $M$ is large compared to $j$. By algebraic $K$-theory it has rank one when $j$ is a positive multiple of four and otherwise has rank zero. Hatcher defined a map $G(N)/O(N)\to \H^s(*)$ (explained in \cite{Igusa1}) and conjectured that it induces a rational isomorphism of $\pi_{4k}$ when $N$ is large compared to $k$. This conjecture was proved in \cite{BokstedtWaldhausen87}.
In \cite{Igusa2}, \cite{Igusa1} the second author used Igusa-Klein higher torsion to give another proof of the conjecture. \\

Here we consider not disks but products of lens spaces with disks. For a finite cyclic group $C_n$ we use an equivariant generalization of Hatcher's construction to give rational generators for $\pi_j\H^s(BC_n)$, the direct limit of $\pi_j \H(L_n^{2N-1}\times D^M)$ as $N$ and $M$ go to infinity. (In fact the limit is attained \cite{0691.57011}, although we do not use that fact.) For $k>0$ let $l_{k,n}$ be the rank of $\pi_{2k}\H^s(BC_n)$. This is $\lfloor\frac{n+2}{2}\rfloor$ if $k$ is even and $\lfloor\frac{n-1}{2}\rfloor=n-\lfloor\frac{n+2}{2}\rfloor$ if $k$ is odd. The rank of $\pi_{2k+1}\H^s(BC_n)$ is zero. 
We introduce a topological monoid $G_n(N)$ of $C_n$-equivariant self homotopy equivalences of $S^{2N-1}$ and show that the space $G_n/U$ obtained as colimit over $N$ of 
$$
G_n(N)/U(N)=fiber(BU(N)\to BG_n(N))
$$
is rationally equivalent to $BU$. We produce a map 
$$
\Delta^a:G_n(N)/U(N)\to \H^s(BC_n)
$$
for each $a\in \Z/n$, and by calculating higher torsion with twisted coefficients we prove:
 
\begin{Theorem*}[Main Theorem] If $N\ge k$ then some set of maps $\Delta^a$ gives an isomorphism 
	$$\pi_{2k}(G_n(N)/U(N))^{l_{k,n}}\otimes \Q\cong \pi_{2k}\H^s(BC_n)\otimes \Q.$$
\end{Theorem*}

In \cite{Igusa1} the second author applied his work on higher torsion and the Hatcher map to classify axiomatic higher torsion invariants for constant coefficient systems.
In \cite{Axioms} the third author applies our work here on twisted higher torsion (with twisting by finite cyclic groups) and generalized Hatcher maps to classify axiomatic higher torsion invariants for coefficient systems in which the twisting is by \emph{any} finite group.

\subsection*{Outline of the paper}

In the first section we review the stable \hcob spaces $\H^s(X)$.\\

In the second section we recall some needed facts about higher torsion, specifically the version developed by the second author with John Klein.\\

In the third section we define the generalized Hatcher maps $\Delta^a$. First we define the spaces $BG_n(N)$ and $G_n(N)/U(N)$ and show that the colimits $BG_n$ and $G_n/U$ are rationally equivalent to a point and $BU$ respectively. Then we explicitly describe the map, or rather weak homotopy class,  $\Delta^1:G_n(N)/U(N)\to \H^s(BC_n).$ We start with a finite complex $B$ and any map $B\to G_n(N)/U(N)$, interpreted as a complex vector bundle of rank $N$ on $B$ together with a $C_n$-equivariant fiber homotopy trivialization of its unit sphere bundle. We use this to produce a map $B\to \H^s(BC_n)$, by making a bundle over $B$ whose fibers are \hcobs of the product of the lens space $L_n^{2N-1}=S^{2N-1}/{C_n}$ and a disk. For general $a\in \mathbb Z/n$ the map $\Delta^a:G_n(N)/U(N)\to \H^s(BC_n)$ is defined by left composition of $\Delta^1$ with a map $\H^s(BC_n)\to \H^s(BC_n)$ induced by the homomorphism $\zeta\mapsto \zeta^a$ from $C_n$ to $C_n$. \\

In the fourth section we complete the proof. The rank of the group $\pi_{2k}\H^s(BC_n)$ is known by algebraic $K$-theory. For each $k>0$ there are $n$ homomorphisms from $\pi_{2k}(G_n(N)/U(N))\otimes\Q\cong\Q$ to $\pi_{2k}\H^s(BC_n)\otimes \Q$, one for each of the maps $\Delta^a$. There are also $n$ homomorphisms $\pi_{2k}\H^s(BC_n)\to\R$, one for each one-dimensional complex representation of $C_n$; they give the higher torsion of lens space \hcob bundles with respect to the coefficient systems determined by the representations. We calculate the resulting $n\times n$ matrix in terms of values of a polylogarithm function at $n$-th roots of unity. By determining the rank of the matrix we verify that the maps $\Delta^a$ give generators for $\pi_{2k}\H^s(BC_n)\otimes \Q$ when $N\ge k$. 

\section{Review of spaces of \hcobs}

We give a definition of stable spaces of \hcobs and recall how they are related to algebraic $K$-theory.
\subsection{The space $\H^s(X)$}\label{hcobspace}

Recall that a cobordism between smooth closed manifolds $M$ and $M'$ is a smooth compact manifold $W$ whose boundary consists of $M$ and $M'$, and that it is called an \hcob if the inclusions $M\into W$ and $M'\into W$ are homotopy equivalences. The definition is extended to the case when $M$ and $M'$ are compact manifolds with boundary and $\partial M\cong \partial M'$. It is customary to let the boundary of $W$ be $M\cup (\partial M\times I)\cup M'$. Here we often use a different convention, letting $\partial W$ be $M\cup M'$, so that $M$ and $M'$ actually intersect in $\partial M=\partial M'$. It is easy to go back and forth between these two set-ups by rounding corners, and we will feel free to do so without comment.\\

When $M$ is fixed and $M'$ is variable we speak of \hcobs on $M$. The $h$-Cobordism Theorem classifies the \hcobs on $M$ up to a diffeomorphism fixed on $M\cup (\partial M\times I)$ (or, in the corner-rounded version, fixed on $M$), provided the dimension of $M$ is at least five.\\

The space of all diffeomorphisms from the trivial cobordism $M\times I$ to itself that are trivial on $(M\times 0)\cup (\partial M\times I)$ (that is, that restrict to the identity there) is called the concordance space and denoted by $\mathcal C(M)$. \\

By an \hcob bundle we mean a smooth fiber bundle whose fibers are \hcobs on a fixed manifold $M$. It is understood that the subbundle whose fibers are the \lq\lq zero ends\rq\rq\ of the \hcobs is a trivial bundle with fiber $M$, and that a definite trivialization is part of the data.\\

There is a classifying space $\H(M)$ for such bundles; for a compact manifold $B$ the homotopy classes of maps $B\to \H(M)$ correspond to isomorphism classes of \hcob bundles with base $B$. The loopspace $\H(M)$ is equivalent to $\mathcal C(M)$. \\

We briefly indicate one construction of $\H(M)$. If $M$ is a submanifold of some $\R^N$, then define $\H(M)$ as a space of submanifolds of $\R^N\times [0,1)$. A point of $\H(M)$ is any subset $W\subset \R^N\times [0,\infty)$ that is an \hcob on $M=M\times 0$ and transverse to $M\times 0$. For any $W\in \H(M)$, let $U_W$ be the set of all those $W'\in \H(M)$ such that there is a diffeomorphism $W\cong W'$ trivial on $M$. The topology of $\H^s(M)$ is defined by letting $U_W$  be open and topologizing it as a quotient of the space of embeddings $W\into R^N\times [0,\infty)$ trivial on $M$. Since the quotient map is a fibration and the space of embeddings is weakly contractible, it follows that the fiber is a delooping of the automorphism group of $W$, in particular a delooping of $\curlyC(M)$ if $W$ is the trivial \hcob.\\

We mention even more briefly an alternative simplicial approach: take $\H(M)$ to be the realization of a simplicial set in which a $k$-simplex is a bundle over $\Delta^k$ of \hcobs on $M$.\\

The stable concordance space
	$$\curlyC^s(M):=\colim_k\  \curlyC(M\times I^k)$$
is defined by iterating the Hatcher suspension map $s:\curlyC(M) \to \curlyC(M × I)$. The concordance stability theorem \cite{0691.57011} states that $s$ induces isomorphisms of homotopy groups up to roughly one third of the dimension of $M$. \\

There is an analogous suspension $s:\H(M)\to \H(M\times I)$. Ignoring corner-rounding issues, we can simply say that it takes $W$, an \hcob of $M$, to $W\times I$, an \hcob of $M\times I$. This leads to the stable \hcob space $\H^s(M)$, an infinite loopspace such that $\Omega\H^s(M)\simeq \curlyC^s(M)$ and $\pi_0\H^s(M)\cong Wh( \pi_1 M)$.\\

In each case (\hcobs and concordances) one can suspend more generally with respect to a vector bundle. If $\xi$ is a vector bundle over $M$ and $D^\xi(M)$ is its closed unit disk bundle then there is a map $\H(M)\to \H(D^\xi(M))$ which in the case of a trivial bundle is the suspension $s: \H(M)\to \H(M\times D^k)$. The colimit over disk bundles is equivalent to the sequential colimit above by cofinality, because every vector bundle is a summand of some trivial bundle.

\subsection{Functoriality and homotopy invariance}

A codimension zero embedding $f:M\into M'$ gives a way of making \hcobs of $M$ into \hcobs of $M'$ as indicated in Figure 1. 
The new \hcob $E'$ is the union $E\cup (M'\times I)$ of the old \hcob $E$ with a trivial cobordism, with each point $x\in M$ in the bottom of $E$ getting identifying with a point $(f(x),1)$ in the top of $M'\times I$. This construction applied fiberwise to a bundle of \hcobs yields a map $\H(M)\to \H(M')$, well defined up to homotopy, depending only on the isotopy class of the embedding $f$.\\

\begin{figure}
\label{hcobsum}

\begin{picture}(230,254)
\setlength{\unitlength}{1cm}
\put(3,0){\includegraphics[width=10cm]{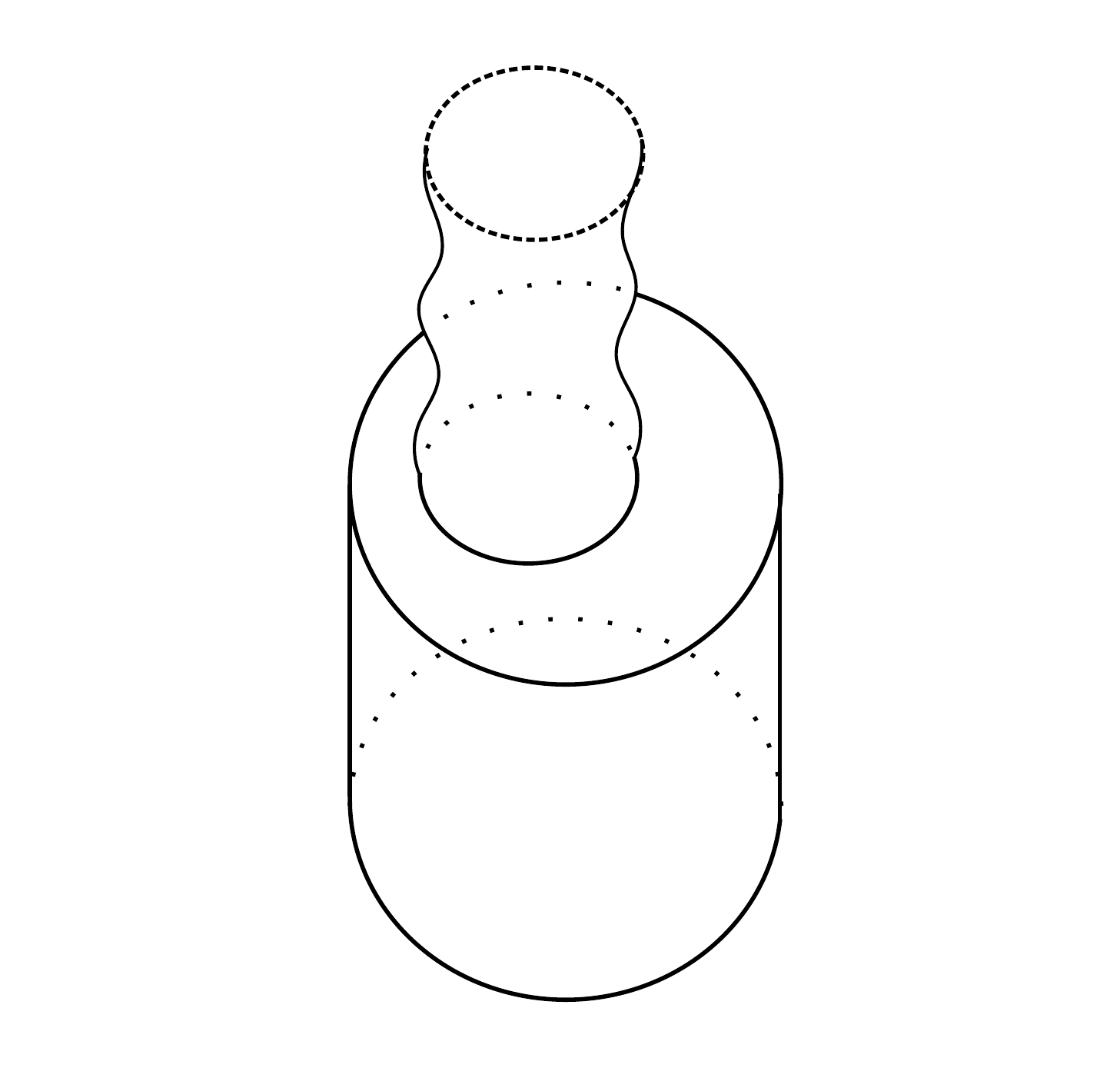}}
\put(9.0,5.3){$M'$}
\put(7.3,5){$f(M)$}
\put(7.3,6.5){$E$}
\put(7.5,2.5){$M'\times I$}
\end{picture}
\caption{Construction of the map $\H(M)\to \H(M')$}
\end{figure}

It follows that a homotopy class of continuous maps $f:M\to M'$ determines a homotopy class of maps $\H^s(f):\H^s(M)\to \H^s(M')$. In fact, choose $N$ large enough so that the homotopy class is represented by a smooth interior embedding $M\into M'\times D^N$, unique up to isotopy. Use the composition 
$$
\H(M)\to \H(D^\nu(M))\to \H(M'\times D^N),
$$
where the first map is Hatcher suspension with respect to the normal bundle $\nu$ of the embedding and the second is induced by the codimension zero embedding $D^\nu(M)\into M'\times D^N$. This produces a functor $\H^s$ from the homotopy category of finite complexes to the homotopy category of infinite loopspaces.\\

The functor can be extended to all CW complexes by defining $\H^s(X)$ as a (homotopy) colimit of $\H^s(Y)$ over finite subcomplexes $Y$ of $X$. For example, for the cyclic group $C_n$ we can take the classifying space $BC_n$ to be an increasing union of lens spaces $L_n^{2N-1}$ as $N\to \infty$ and take $\H^s(BC_n)$ to be the colimit of $\H^s(L_n^{2N-1})$. 

\subsection{Relation with algebraic $K$-theory}

Waldhausen's functor $A$ from spaces to spectra is related to smooth manifolds by a splitting:
$$A(X)\simeq  \Sigma^\infty(X_+)\times Wh^{\mathit{diff}}(X),$$
where $\Omega^\infty Wh^{\mathit{diff}}(X)$ is a delooping of $\H^s(X)$. A full proof of this is given in \cite{zbMATH06165458}.\\

On the other hand, $A(X)$ is also closely related to the algebraic $K$-theory of rings. In particular for a based path-connected space $X$ there is a natural map 
$$A(X)\to K(\Z[\pi_1X])$$
to the $K$-theory spectrum of the group ring of the fundamental group. When $X$ has trivial higher homotopy groups then this map is a rational homotopy equivalence (see \cite{W85} p. 399, paragraph following Corollary 2.3.8). Thus for a discrete group $\Gamma$ we have
$$
K_i(\Z[\Gamma])\otimes \Q\cong H_i(B\Gamma;\Q)\oplus \pi_{i-1}\H^s(B\Gamma) \otimes \Q,
$$
and when $\Gamma$ is finite we simply have 
\begin{equation}\label{KandH}
K_i(\Z[\Gamma])\otimes \Q\cong \pi_{i-1}\H^s(B\Gamma)\otimes\Q.
\end{equation}

The left hand side is known (see \cite{1189.19001}):

\begin{Theorem} If $\Gamma$ is a finite group and $i\ge 2$ then
	$$K_i(\Z[\Gamma])\otimes \Q\cong\left\{
		\begin{array}{lll}
			\Q^r	&	\textnormal{if}	&	i\equiv 1\mod 4\\
			\Q^c	&	\textnormal{if}	&	i\equiv 3\mod 4\\
			0	&	\textnormal{if}	&	i\equiv 0\mod 2
		\end{array}\right. ,
	$$
where $r$ is the number of irreducible real representations of $\Gamma$ and $c$ is the number of irreducible real representations of complex type.
\end{Theorem}

If $\Gamma$ is finite and abelian, then $r+c=|\Gamma|$ and $r-c=|\Gamma/2|$. In particular when $\Gamma$ is $C_n$ we have $r=\lfloor\frac{n+2} 2\rfloor$ and $c=\lfloor\frac{n-1} 2\rfloor$.

\section{Review of higher torsion invariants}

Higher torsion invariants have been developed by J. Wagoner, J. R. Klein, M. Bismut, J. Lott, W. Dwyer, M. Weiss, E. B. Williams, S. Goette, the second author and many others (\cite{0793.19002}, \cite{0837.58028}, \cite{1077.19002}, \cite{1071.58025}, \cite{Igusa2}). Here we use the Igusa-Klein version, which is defined using parametrized Morse theory. We recall some of its properties. \\

\subsection{General properties}

Let $F\into E\to B$ be a smooth fiber bundle, where $B$ and $E$ are compact smooth manifolds and $F$ is an orientable smooth manifold. If $F$ has a boundary then the fiberwise or vertical boundary $\partial^v E\to B$ is a subbundle of $E.$  We assume that $B$ is connected. Let $\F$ be a local system on $E$ given by a finite-dimensional unitary complex representation of the fundamental group.\\

Assume also that either
\begin{enumerate}
\item $\pi_1B$ acts trivially on $H_*(F;\F)$ or
\item $\F$ is a one dimensional representation and $H_*(F;\F)$ is unipotent as a $\pi_1B$ module.
\end{enumerate}
Then the Igusa-Klein torsion is defined (by 5.7.5 and 3.6.12 in \cite{Igusa2}). It is a cohomology class of the base space:
	$$\tau_k^{\mathit{IK}}(E;\F)\in H^{2k}(B;\R).$$
In this paper we will usually assume Condition (ii). In the twisted Hatcher examples, both conditions will hold.\\

We now list some relevant properties of Igusa-Klein torsion. Details can be found in \cite{Igusa2}, \cite{Igusa1}, and \cite{Axioms}.\\

\textit{Naturality.} Torsion is a characteristic class: for a map $f:B'\to B$ and a bundle $E\to B$ with local system $\F$ on $E$ we have
	$$\tau_{k}^{\mathit{IK}}(f^*E;f^*\F)=f^*\tau_{k}^{\mathit{IK}}(E; \F)\in H^{2k}(B';\R).$$
Note that naturality implies that for trivial bundles (bundles pulled back from a one-point space) the torsion is zero. \\

\textit{Geometric additivity.} 
Suppose that $E_1\cup E_2\to B$ is a bundle such that $E_1\to B$, $E_2\to B$, and $E_1\cap E_2$ are subbundles. We assume that the fibers are such that the manifold $F$ is the union of manifolds $F_1$ and $F_2$ along a codimension zero submanifold of their boundaries. Then
	$$\tau_k^{\mathit{IK}}(E_1\cup E_2; \F)=\tau_k^{\mathit{IK}}(E_1; \F_1)+\tau_k^{\mathit{IK}}(E_2; \F_2)-\tau_k^{\mathit{IK}}(E_1\cap E_2; \F_{12}),$$
if all four terms are defined. \\

\textit{Stability for disk bundles.} Let $X\to E$ be a linear disk bundle (the bundle of closed disks associated with a vector bundle) and suppose that $E\to B$ is a bundle for which torsion is defined. Then 
\begin{eqnarray}\label{IKstability}\tau^{\mathit{IK}}_{k}(X; \F)=\tau^{\mathit{IK}}_{k}(E; \F).\end{eqnarray}

\textit{Additivity for coefficients.} For two local systems $\F_1$ and $\F_2$ on $E$ we have 
	$$\tau_{k}^{\mathit{IK}}(E; \F_1\oplus \F_2)=\tau_{k}^{\mathit{IK}}(E; \F_1)+\tau_{k}^{\mathit{IK}}(E; \F_2).$$
	
\textit{Transfer for coefficients.} If bundles $\tilde E\to B$ and $E\to B$ are related by a finite covering $\pi:\tilde E\to E$, then for every local system $\tilde \F$ on $\tilde E$ we have
	$$\tau_k^{\mathit{IK}}(\tilde E; \tilde \F)=\tau_k^{\mathit{IK}}(E;\pi_*\tilde \F),$$
where $\pi_*$ denotes the push-down operator for local systems.

\begin{Remark} The second author investigated axiomatic higher torsion, taking some of the formal properties above (and others not listed here) as axioms for invariants of unipotent manifold bundles. Considering only the case of untwisted coefficient systems, he showed that the space of higher torsion invariants is two dimensional and spanned by the Igusa-Klein torsion and the Miller-Morita-Mumford class, which is essentially a Chern class \cite{Igusa1}. The third author investigates axiomatic higher torsion for twisted coefficients and uses the results of this paper to obtain an analogous classification for twisting by any finite group \cite{Axioms}.
\end{Remark}

\subsection{Torsion of bundles of \hcobs}\label{hcobtorsion}

Let $W\into E\to B$ be a bundle of \hcobs on $M$, with trivial subbundle $M\into M\times B\to B$ (the zero end of the cobordism). Since the inclusion $M\times B\into E$ is a homotopy equivalence, a local system $\F$ on $M$ yields a local system on $E$, which we will also call $\F$. The torsion $\tau^{\mathit{IK}}(E;\F)$ is defined, since $\pi_1B$ acts trivially up to homotopy on the fiber.\\

Given such a bundle $W\into E\to B$, the Hatcher suspension discussed in \ref{hcobspace} produces a bundle $s(W)\into E'\to B$ of \hcobs of $M\times I$. Since $E'$ is a (trivial) $1$-disk bundle over $E$, Equation \eqref{IKstability} applies and shows that the new bundle has the same torsion as the old one. Therefore we may speak of the torsion of a stable \hcob bundle; together, a map $\varphi:B\to \H^s(M)$ and a local system $\F$ on $M$ yield an element of $H^{2k}(B;\R)$. \\

Since it is natural in $B$, this class may be described in terms of a universal example, an element $\tau_{k,M\F}^h\in H^{2k}(\H^s(M);\R)$. Thus we write $\varphi^*\tau^h_{k,M,\F}$ for $\tau_k^{\mathit{IK}}(E;\F)$ if the \hcob bundle $E\to B$ is given stably by the map $\varphi:B\to \H^s(M)$. \\

The torsion of stable \hcob bundles is natural with respect to $M$ as well. In fact, if $f:M\to M'$ is a map between manifolds and $\F$ is a local system on $M'$ then 
	$$\tau_{k,M,f^*\F}^h=\H^s(f)^*\tau_{k,M',\F}^h.$$
To see this, assume without loss of generality that $f$ is a codimension zero embedding. Consider any bundle $E\to B$ of \hcobs of $M$, and let $E'\to B$ be the resulting bundle of \hcobs of $M'$ (Figure \ref{hcobsum}). The assertion is that  $\tau_k^{\mathit{IK}}(E'; \F)=\tau_k^{\mathit{IK}}(E; f^*\F)$. Additivity gives
	\begin{eqnarray*}	
		\tau_k^{\mathit{IK}}(E'; \F)	&	=	&	\tau_k^{\mathit{IK}}(E; f^*\F)  + \tau_k^{\mathit{IK}}(M'\times I\times B;\F) - \tau_k^{\mathit{IK}}(M'\times B;\F),
	\end{eqnarray*}
and the last terms are zero because the bundles are trivial.\\

We can extend the definition from manifolds to spaces, defining $\tau_{k,X,\F}^h\in H^{2k}(\H^s(X);\R)$ when $\F$ is a local system on $X$. There is a unique way to do this such that naturality continues to hold, in other words such that
\begin{eqnarray}\label{othernat}\tau_{k,X,f^*\F}^h=\H^s(f)^*\tau_{k,X',\F}^h\end{eqnarray}
for $f:X\to X'$.

\subsection{Polylogarithms and the torsion of linear lens space bundles}\label{lenstorsion}

We recall a formula ((\ref{formula}) below, taken from \cite{Axioms} and \cite{Igusa2}) for the torsion of a bundle of lens spaces determined  by a complex vector bundle. This will be needed below in computing the torsion of some bundles of \textit{h}-cobordisms. \\

Let $\xi$ be a complex vector bundle of rank $N$ on a compact manifold $B$. Let $S(\xi)$ be the unit sphere bundle, equipped with the canonical free action of the group $C_n$. It becomes important to keep track of generators at this point. From now on we identify $C_n$ with the group of complex $n$-th roots of unity. The orbit space $S(\xi)/C_n$ is a bundle of $(2N-1)$-dimensional lens spaces over $B$.  \\

For any complex $n$-th root of unity $\zeta$, define the one-dimensional local system $\F_{\zeta}$ on $S(\xi)/C_n$ to be $S(\xi)\times_{C_n} \C\to S(\xi)/C_n,$ where the generator $\zeta_n=e^{2\pi i/n}\in C_n$ acts on $\C$ by multiplication by $\zeta.$ \\

Here is the torsion formula:
\begin{eqnarray}\label{formula}\tau_{k}^{\mathit{IK}}(S(\xi)/C_n; \F_{\zeta})=-{n^k}L_{k+1}(\zeta)ch_{2k}(\xi).
\end{eqnarray}
The function $L_{k+1}$ is the real polylogarithm defined by
	$$L_{k+1}(z)=\Re\left(\frac 1 {i^k} \L_{k+1}(z)\right),$$
where $\L_{k+1}$ is the complex polylogarithm
		$$\L_{k+1}(z)= \sum_{m=1}^\infty \frac{z^m}{m^{k+1}}.$$
		
The formula \eqref{formula} is proved as Theorem 2.8.4 in \cite{Igusa2}. We recall some parts of the argument. \\

Using the splitting principle and the naturality of torsion one can reduce to the case when $\xi$ is a direct sum of line bundles.\\

One can then reduce further to the case of a line bundle using geometric additivity and other properties of torsion. \\

Let $\xi$ be a complex line bundle. Transfer and additivity for coefficients together imply
\begin{eqnarray}\label{one}\tau_{k}^{\mathit{IK}}(S(\xi)/C_n; \F_{\zeta^m})=\sum_{\omega^m=1}\tau_{k}^{\mathit{IK}}(S(\xi)/C_{mn}; \F_{\omega\zeta}).
\end{eqnarray}
We also have (by naturality)
\begin{eqnarray}\label{two}\tau_k^{\mathit{IK}}(S(\xi)/C_n;\F_{\zeta})=m^{-k}\tau_k^{\mathit{IK}}(S(\xi)/C_{mn};\F_{\zeta})
\end{eqnarray}
for every $m\ge 1$.
This yields a function $f_\tau$ from the set of all complex roots of unity to $\H^{2k}(B;\R),$ by putting 
	$$f_\tau(\zeta):=n^{-k}\tau_k^{\mathit{IK}}(S(\xi)/C_n;\F_{\zeta}).$$
Here $\zeta$ is a (not necessarily primitive) $n$-th root of unity. This quantity is independent of $n$ by Equation \eqref{two}. \\

Now Equation \eqref{one} says that the function $f_\tau$ satisfies the \it Kubert identity\rm\  for every $m>0$:
\begin{eqnarray}\label{Kubert}f(z)=m^k\sum_{w^m=z}f(w).
\end{eqnarray}
It is shown in \cite{Igusa2} that $f_\tau$ extends continuously to the circle $|z|=1$. The function $\L_{k+1}$ also satisfies \eqref{Kubert}.
By considering Fourier coefficients one sees that the space of all continuous functions on the circle that satisfy \eqref{Kubert}
is two-dimensional, spanned by $\L_{k+1}$ and its complex conjugate. 
In view of the behavior of torsion under complex conjugation of coefficient systems, we also have 
	$$\tau_k^{\mathit{IK}}(S(\xi)/C_n;\F_{\zeta^{-1}})=(-1)^k\tau_k^{\mathit{IK}}(S(\xi)/C_n;\F_{\zeta}),$$
and it follows that 
	$$\tau_{k}^{\mathit{IK}}(S(\xi)/C_n; \F_{\zeta})=K{n^k}L_{k+1}(\zeta)c_1(\lambda)^k,$$
where $c_1(\lambda)$ is the first Chern class of the line bundle $\lambda$ and $K$ is some constant depending only on $k$. Evaluation of $K$ in a special case (Theorem 7.10.2 in \cite{Igusa2}) gives (4). (See \cite{Axioms} and \cite{Igusa2} for details).\\

\section{Construction of the generalized Hatcher maps}

\subsection{The space $G_n/U$}

Let $C_n$ act in the usual way on $\C^N$ and on its unit sphere $S^{2N-1}$. Let $\M_n(N)$ be the topological monoid of $C_n$-equivariant maps $S^{2N-1}\to S^{2N-1}$, and let $G_n(N)\subset \M_n(N)$ consist of those equivariant maps which are homotopy equivalences. (Note that because the group is acting freely on the sphere every equivariant map that is nonequivariantly a homotopy equivalence is in fact invertible up to equivariant homotopy.) There is an equivariant Freudenthal suspension map from $\mathcal M_n(N)$ to $ \mathcal M_n(N+1)$, which takes $G_n(N)$ into $G_n(N+1)$. Define 
	$$\mathcal M_n:=\lim_\rightarrow \mathcal M_n(N).$$
	$$G_n:=\lim_\rightarrow G_n(N).$$
	
The topological monoid $G_n(N)$ is grouplike (that is, the monoid $\pi_0(G_n(N))$ is a group), and therefore it gets a connected delooping in the usual way. The space $BG_n(N)$ is a classifying space for fibrations with free action of $C_n$ such that each fiber is equivariantly homotopy equivalent to $S^{2N-1}$. In the limit we have $BG_n$, a connected delooping of $G_n$. \\

The Lie group $U(N)$ is contained in $G_n(N)$. Let $G_n(N)/U(N)$ be the homotopy fiber of $BU(N)\to BG_n(N)$, and let $G_n/U$ be the homotopy fiber of $BU\to BG_n$. 

\begin{Lemma} \label{spaceG}	For $0\le m<2N-1$ the group $\pi_m G_n(N)$ is finite and independent of $N$. Therefore $\pi_m G_n$ is finite for all $m\ge 0$, and for every $k>0$ the map $\pi_{2k}(G_n/U)\to\pi_{2k}(BU)\cong \Z$ has finite kernel and cokernel.
\end{Lemma}

We first sketch a proof by obstruction theory. Let $L_n^{2N-1}$ be the lens space $S^{2N-1}/C_n$. Points in $\mathcal M_n(N)$ can be made to correspond to sections of a certain bundle whose base is $L_n^{2N-1}$ and whose fibers are $(2N-1)$-dimensional spheres (see below for details). Thus there is a spectral sequence for its homotopy groups, with
$$E^2_{i,j}=H^{-i}(L_n^{2N-1};\pi_jS^{2N-1}).$$
Note that $E^2_{i,j}$ is zero except when $0\le -i\le 2N-1$ and $j\ge 2N-1$, and that modulo finite groups there is nothing in $E^2$ except two copies of $\Z$: one at $(i,j)=(0,2N-1)$) and one at $(i,j)=(1-2N,2N-1)$. Since each total degree in the range $0<i+j<2N-1$ has only finite groups, the group $\pi_m\mathcal M_n(N)=\pi_mG_n(N)$ is finite for $m$ in this range. The monoid $\pi_0\mathcal M_n(N)$ is infinite, but the group $\pi_0G_n(N)$ (the set of invertible elements in that monoid) can be seen to be finite, in fact of order at most two. \\

The interested reader might work out the details of the argument sketched above. This kind of spectral sequence is constructed in \cite{Schultz}. We ignored base points, and we ignored issues about exact sequences of sets as opposed to groups, nor did we verify the independence of $N$.\\

We take another approach that yields a more detailed statement:\\

\begin{Lemma}\label{better}
The map $\pi_i\mathcal M_n(N)\to \pi_i\mathcal M_n$ is an isomorphism as soon as $2N>i$. For $i>0$ we have $\pi_i\mathcal M_n\cong \pi_i^S((BC_n)_+)$ for any basepoint. We also have $\pi_0\mathcal M_n\cong\Z$ as a set. The monoid structure on $\pi_0$  is given by $x*y=x+y-nxy$.  
\end{Lemma}

Note that this implies that the group $\pi_0G_n$ is trivial when $n\ge 3$, and that it has order two when $n$ is $1$ or $2$. Since $G_n(N)$ is an open and closed subset of $\mathcal M_n(N)$, $\pi_iG_n$ is the same as the finite group $\pi_i\mathcal M_n$ for $i\ge 1$.\\

\begin{proof}
For a smooth closed manifold $M$ and a point $p\in M$, let $T_pM$ be the tangent space and let $S^{T_pM}$ be its one-point compactification. This is the fiber over $p$ of a sphere-bundle $S^{TM}$. When $M$ is a sphere $\Sigma$, we may identify $S^{T_p\Sigma}$ with $\Sigma$ itself, letting $0$ correspond to $p$ and $\infty$ to $-p$. A map $\Sigma\to \Sigma$ corresponds then to a section of the bundle $S^{T\Sigma}$, with the identity corresponding to the section $0$. If $M$ is the orbit space for a free action of a finite group $\Gamma$ on $\Sigma$ by linear isometries, with quotient map $\pi:\Sigma\to M$, then we can identify $S^{T_{\pi(p)}M}$ with $\Sigma$, and in such a way that $\Gamma$-equivariant maps $\Sigma\to \Sigma$ correspond to sections of $S^{TM}$, again with the identity map corresponding to $0$ and the antipodal map to $\infty$. As a special case, we can interpret $\mathcal M_n(N)$ as the space of sections of $S^{TL_n(N)}$.\\

Let $\Gamma(S^{TM})$ be the space of sections.

Claim: For $i<m-1=\dim(M)-1$ we have $\pi_i\Gamma(S^{TM})\cong\pi_i^S(M_+)$.\\

In particular this yields the desired description of $\pi_i\mathcal M_n(N)$ for $i<2N-1$.\\

Proof of Claim: First make a Thom-Pontryagin isomorphism between $\pi_i\Gamma (S^{TM})$ and a set of bordism classes, namely classes of closed $i$-dimensional manifold $W\subset \R^i\times M$ equipped with an isomorphism between the normal bundle and $TM$; bordisms are in $I\times \R^i\times M$. (An element of $\pi_i\Gamma(S^{TM})$ is represented by a family $\sigma_u$ of sections of $S^{TM}$ parametrized by $u\in D^i$ with $\sigma_u=\infty$ for $u\in \partial D^i$. When this is in general position, it determines a manifold $W=\lbrace (u,x): \sigma_u(x)=0\rbrace$. The proof that this map from homotopy classes to bordism classes is well-defined and in fact bijective follows the usual pattern.) \\

Now consider the direct limit of 
$$\pi_i\Gamma (S^{TM})\to \pi_{i+1}\Gamma (S^{TM+1})\to \dots \to \pi_{i+k}\Gamma (S^{TM+k})\to \dots,
$$
where the first map is induced by the canonical map 
$$
\Gamma (S^{TM})\to \Omega \Gamma(S^{TM+1})
$$
and so on.
The $k$-th group corresponds to bordism classes of $i$-dimensional manifolds in $\R^{i+k}\times M$ with normal bundle $TM+k$. The direct limit is the $i$th framed bordism group of $M$. On the other hand the maps are all isomorphisms if $i<m-1$, by the Freudenthal Suspension Theorem applied to the maps $S^{T_pM+k}\to \Omega S^{T_pM+k+1}$. Thus
\[
	\pi_i\Gamma(S^{TM})\cong \Omega_i^{fr}(M)\cong\pi_i^S(M_+)
\]
if $i<m-1$.
\qed\\

For $i>0$ the isomorphism given by the Claim is a group isomorphism. In the case when $M$ is the lens space $L_n^{2N-1}$ it remains to describe the monoid structure on $\pi_0\mathcal M_n$. Let $f:S^{2N-1}\to S^{2N-1}$ be an equivariant map in the class corresponding to the integer $x$. The fixed points of $f$ occur in orbits of $n$ points, one orbit for each zero of the corresponding section of $S^{L_n^{2N-1}}$. Thus if these zeroes occur transversely and $x$ is the signed number of zeroes then $nx$ is the Lefschetz number of $f$ and $1-nx$ is the degree. The formula $x+y-nxy$ now follows since degree is multiplicative for composition.
\end{proof}

\subsection{The maps $\Delta^a:G_n(N)/U(N)\to \H^s(BC_n)$}\label{construction}

We now construct, for every $n>0$, a family of maps $\Delta^a:G_n(N)/U(N)\to \H^s(BC_n)$ indexed by $a\in \Z/n$. In the $n=1$ case the map is the composition of Hatcher's $G(2N)/O(2N)\to \H^s(*)$ with the canonical map $G(2N)/U(N)\to G(2N)/O(2N).$ We will first construct $\Delta^1$ and then obtain the other maps $\Delta^a$ from it using maps $BC_n\to BC_n$ which depend on $a$.\\

We specify the homotopy class $\Delta^1$ by giving, for every finite complex $B$, a map $\Delta^1_*:[B,G_n(N)/U(N)]\to[B,\H^s(BC_n)]$, natural in $B$. 
We may take $B$ to be a smooth compact manifold. \\

A map $\xi: B\to G_n(N)/U(N)$ corresponds to a rank $N$ complex vector bundle $\xi$ over $B$ together with a $C_n$-equivariant fiber homotopy trivialization of the sphere bundle $S^{2N-1}(\xi)$. 
This means that there is a commutative diagram
	$$\xymatrix{
		S^{2N-1}(\xi)\ar[r]^{\tilde H}\ar[d]	&	S^{2N-1}\times B\ar[d]\\
		B\ar[r]^=										&	B
		}
	$$
such that the map $\tilde H$ induces a $C_n$-equivariant homotopy equivalence in every fiber. Consider the orbit space of $S^{2N-1}(\xi)$ under the $C_n$-action, a lens space bundle $L_n^{2N-1}(\xi)\to B$. Denote the fiber over $t\in B$ by $L_{n,t}^{2N-1}(\xi)$. The map $\tilde H$ induces a homotopy equivalence 
	$$H_t:L_{n,t}^{2N-1}(\xi)\to L_n^{2N-1}.$$

We are trying to describe a bundle of \hcobs over $B$. In the interest of clarity, we will continue to explain the construction in terms of the fiber over each point $t\in B$.
If $M$ is large enough then we can lift the family of maps $H_t$ to a family of smooth embeddings
	$$\bar H_t:L_{n,t}^{2N-1}(\xi)\to L_n^{2N-1}\times D^M.$$
Let $\nu_t$ be the normal bundle of $\bar H_t$, a vector bundle over $L_{n,t}^{2N-1}(\xi)$. Make a fiberwise tubular neighborhood, an embedding
	$$G_t:D^M(\nu_t)\into L_n^{2N-1}\times D^M,$$
depending smoothly on $t$. Let $E'_t(\xi)$ be the closure of the complement of the image of $G_t$. \\

The manifold $E'_t(\xi)$ is an \hcob between (the total space of) the sphere bundle $S^{M-1}(\nu_t)$ over $L_{n,t}^{2N-1}(\xi)$ and $L_n^{2N-1}\times S^{M-1}$. Cross it with $I$ and attach the result to $L_n^{2N-1}\times D^M$ as in Figure 1 by embedding $I\times S^{M-1}(\nu_t)\into L_n^{2N-1}\times D^M$. (The embedding is the restriction of $G_t$ to a collar of the boundary.) This makes a bundle over $B$ whose fibers are \hcobs of $L_n^{2N-1}\times D^M$.\\

The resulting element of $[B,\H^s(BC_n)]$ is called $\Delta^1_*\xi$. It is clear that this depends only on the initial map $B\to G_n(N)/U(N)$, so that we have now defined a map $\Delta^1:G_n(N)/U(N)\to \H^s(BC_n)$.\\

More generally for $a\in \Z/n$ we define the generalized Hatcher map $\Delta^a$ as the composition
	$$\Delta^a:G_n(N)/U(N)\xrightarrow{\Delta^1}\H^s(BC_n)\xrightarrow{\H^s(Ba)}\H^s(BC_n),$$
where $Ba:BC_n\to BC_n$ is induced by the homomorphism $a:\zeta\mapsto \zeta^a$ from $C_n$ to itself.\\

\begin{Remark}Here is an equivalent description of the \hcob bundle that defines $\Delta^1_*\xi$: let $E_t\subset L_n^{2N-1}\times D^M\times I$ be given by 
	$$E_t=(G_t(D^M(\nu_t))\times I)\cup (L_n^{2N-1}\times D^M\times ([0,1/3]\cup [2/3,1])).$$
viewed as a cobordism of $L_n^{2N-1}\times D^M\times 0$. To see that it is equivalent to $E_t$ as first described, note that  the subset
	$$(G_t(D^M(\nu_t))\times I)\cup (L_n^{2N-1}\times D^M\times ([0,1/3]))$$
is a trivial cobordism of $L_n^{2N-1}\times D^M\times 0$ and that the remaining part, $E'_t(\xi)\times [2/3,1]$, is a copy of $E'_t(\xi)\times I$ attached in the right way. 
\end{Remark}



\section{Proof of the Main Theorem}

For each $a\in \Z/n$ the generalized Hatcher map $\Delta^a$ gives a homomorphism 
	$$\Delta^a_\ast: \pi_{2k}(G_n(N)/U(N))\otimes \Q\to  \pi_{2k}\H^s(BC_n)\otimes \Q,$$
where the left hand side is $1$-dimensional and the dimension  $l_{k,n}$ of the right hand side is $\lfloor {\frac{n-1}2}\rfloor$ if $k$ is odd and $\lfloor\frac{n+2}2\rfloor$ if $k$ is even. 

\begin{Theorem}[Main Theorem]\label{maintheorem} If $N\ge k>0$ then some set of maps $\Delta^a$ produces an isomorphism
	$$\pi_{2k}(G_n(N)/U(N))^{l_{k,n}}\otimes \Q\cong \pi_{2k}\H^s(BC_n)\otimes \Q.$$

\end{Theorem}

\begin{Remark} More precisely, an isomorphism is given by using the maps $\Delta^a$ for $0<a\le \lfloor{ \frac{n-1}2}\rfloor$ if $k$ is odd, and for $0\le a\le \lfloor{ \frac{n}2}\rfloor$ if $k$ is even. 
\end{Remark}

To prove the theorem we have to show that certain elements of $\pi_{2k}\H^s(BC_n)$ are rationally linearly independent. We do this by determining the Igusa-Klein torsion of the corresponding \hcob bundles over $S^{2k}$.





\subsection{Calculation of the torsion of $\Delta^a_*\xi$}

For each $b\in \Z/n$ the class $\tau^h_{k,BC_n,\F_{\zeta_n^b}}\in H^{2k}(\H^s(BC_n);\R)$ associated with the local system $\F_{\zeta_n^b}$ on $BC_n$ gives a homomorphism $\pi_{2k}\H^s(BC_n)\to H^{2k}(S^{2k};\R)\cong \R$. We evaluate the composition of this with the map $\Delta^a_\ast$ for every $a$ and $b$. 

\begin{Lemma}\label{actionontorsion} For any $\xi\in \pi_{2k}(G_n(N)/U(N))$ we have 
$$\langle\tau^h_{k,BC_n,\F_{\zeta_n^b}},\Delta^a_*\xi\rangle=n^kL_{k+1}(\zeta_n^{ab})\langle ch_{2k}(\xi),\lbrack S^{2k}\rbrack\rangle.$$
\end{Lemma}

\begin{proof} 

We first reduce to the case $a=1$ by observing that 
	$$\langle\tau^h_{k,BC_n,\F_{\zeta_n^b}},\Delta^a_*\xi\rangle=\langle\Delta^{a*}\tau^h_{k,BC_n,\F_{\zeta_n^b}},\xi\rangle=\langle\Delta^{1*}\tau^h_{k,BC_n,\F_{\zeta_n^{ab}}},\xi\rangle.$$
The second equation follows from the naturality statement \eqref{othernat} at the end of \ref{hcobtorsion}, with $f:X\to X'$ being the map $Ba:BC_n\to BC_n$ and $\F$ being the local system $\F_{\zeta_n^b}$ on $BC_n$. Note that $(Ba)^*\F_{\z_n^b}=\F_{\zeta_n^{ab}}$ and $\Delta^{1*}\circ (Ba)^*=\Delta^{a*}$. \\

For the case $a=1$, let $E(\xi)$ be the bundle of \hcobs representing $\Delta^1_*\xi$, as constructed in \ref{construction}. \\

Pulling back the class $\tau^h_{k,BC_n,\F_{\zeta_n^b}}$ by the map $\Delta^1_*\xi:S^{2k}\to \H^s(BC_n)$ gives $\tau_k^{\mathit{IK}}(E;\F_{\z_n^b})$, so that
	$$\langle\tau^h_{k,BC_n,\F_{\zeta_n^b}},\Delta^1_*\xi\rangle=\langle \tau_k^{\mathit{IK}}(E(\xi);\F_{\z_n^b}),\lbrack S^{2k}\rbrack\rangle.$$
To calculate the right-hand side we use additivity twice. \\

$E(\xi)$ is the union, along $S^{M-1}(\nu)$, of $I\times E'(\xi)$ and a trivial bundle, so we have (for any coefficient system) 
	$$\tau_k^{\mathit{IK}}(E(\xi);\F)=\tau_k^{\mathit{IK}}(I\times E'(\xi);\F)-\tau_k^{\mathit{IK}}(S^{M-1}(\nu);\F)=\tau_k^{\mathit{IK}}(E'(\xi);\F)-\tau_k^{\mathit{IK}}(S^{M-1}(\nu);\F).$$
Since $E'(\xi)$ is the closed complement of the embedded image of $D^M(\nu)$ in a trivial bundle, we have
	$$\tau_k^{\mathit{IK}}(E'(\xi);\F)=-\tau_k^{\mathit{IK}}(D^M(\nu);\F)+\tau_k^{\mathit{IK}}(S^{M-1}(\nu);\F).$$
Because $D^M(\nu)$ is a linear disk bundle over the lens-space bundle $L_n^{2N-1}(\xi)$, we also have 
	$$\tau_k^{\mathit{IK}}(D^M(\nu);\F)=\tau_k^{\mathit{IK}}(L_n^{2N-1}(\xi);\F).$$
Putting this together and specializing to $\F=\F_{\z^b}$, we find that 
	$$\tau_k^{\mathit{IK}}(E(\xi);\F_{\z^b})=-\tau_k^{\mathit{IK}}(L_n^{2N-1}(\xi);\F_{\z^b}),$$ 
which is equal to $n^kL_{k+1}(\zeta_n^{b})ch_{2k}(\xi)$ by \ref{lenstorsion}.  
\end{proof}

If $N\ge k$ then some element $\xi\in\pi_{2k}(G_n(N)/U(N))$ is rationally nontrivial and the Chern character $ch_{2k}(\xi)\in H^{2k}(S^{2k};\R)$ is nonzero. Thus to complete the proof of the Main Theorem it suffices to show that the matrix of real numbers 
	$$(L_{k+1}(\zeta_n^{ab}))_{a,b=1}^n.$$
has rank $\lfloor \frac{n+2}2\rfloor$ if $k$ is even and $\lfloor \frac{n-1}2\rfloor$ if $k$ is odd.

\subsection{Proof of linear independence}

The real matrix above is (up to sign) the real or imaginary part of the complex matrix $T$ given by
	$$T_{a,b}=\L_{k+1}(\zeta_n^{ab}),$$
according as $k$ is even or odd. \\

Note that $\bar{\L_{k+1}(z)}=\L_{k+1}(\bar z)$, so that $\bar T_{a,b}=T_{a,n-b}$. That is, $\bar T=TS$ where $S$ is the $n\times n$ permutation matrix defined by $S_{a,b}=1$ if $n$ divides $a+b$ and otherwise $S_{a,b}=0$. Assuming $T$ is invertible, we have
	$$\rk(\Re(T))=\rk(T+\bar T)=\rk(T+TS)= \rk(1+S)=\lfloor \frac{n+2}2\rfloor$$
and
	$$\rk(\Im(T))=\rk(T-\bar T)=\rk(T-TS)= \rk(1-S)=\lfloor \frac{n-1}2\rfloor.$$

Therefore the Main Theorem will follow from:

\begin{Lemma}\label{lastlemma}
For every $n\ge 1$ and $k\ge 1$ the matrix $T$ is invertible.
\end{Lemma}

\begin{Remark}\label{induction seed}
In the case $n=1$ this is the statement that $\L_{k+1}(1)$, which is the value of Riemann's $\zeta$ at $k+1>1$, is not zero. In the general case the statement will boil down, after some manipulations involving the Kubert identity and induction on $n$, to the non-vanishing of Dirichlet $L$-functions at $k+1>1$.
\end{Remark}

Let $(\Z/n)^*$ be the unit group of the ring $\Z/n$. Lemma \ref{lastlemma} will follow from a statement about a submatrix of $T$:

\begin{Lemma}\label{relprime}
For every $n\ge1$ and $k\ge 1$ the matrix $(T_{a,b})_{a,b\in (\Z/n)^*}$ is invertible.
\end{Lemma}

\begin{proof}[Proof of Lemma \ref{lastlemma} assuming Lemma \ref{relprime}] The proof is by induction on $n$. In the $n=1$ case Lemma \ref{relprime} and Lemma \ref{lastlemma} are the same statement, so there is nothing to prove. In the general case we deduce \ref{lastlemma} for $n$ from \ref{relprime} for $n$ and \ref{lastlemma} for all values less than, specifically for all $n/p$ when $p$ is a prime dividing $n$. 

Suppose that $(y_a)_{a\in\Z/n}$ are complex numbers such that 
$$
\sum_a y_aT_{a,b}=0$$
for all $b\in \Z/n$. We must show that $y_a=0$ for all $a$. If we can show that $y_a=0$ whenever $a\notin (\Z/n)^*$, then by Lemma \ref{relprime} the rest will follow.  \\

Thus for every prime factor $p$ of $n$ we must show that $y_a=0$ when $p$ divides $a$ . Write $n=pn'$.\\

For any $a$ such that $p$ does \emph{not} divide $a$, the Kubert identity (Equation \ref{Kubert}, \ref{lenstorsion}) for $\L_{k+1}$ implies that 
$$
\sum_{j=1}^{p}\ T_{a,b+jn'}={p^{-k}}T_{pa,b}.
$$
For any $a$ such that $p$ does divide $a$, we have $T_{a,b+jn'}=T_{a,b}$, and therefore
$$
\sum_{j=1}^{p}\ T_{a,b+jn'}=pT_{a,b}.
$$
Multiplying by $y_a$ and summing over all $a\in \Z/n$, we find that
\[
	\sum_{(p,a)=1} y_ap^{-k}T_{pa,b}+\sum_{p|a}\ y_apT_{a,b}=0.
\]
Since $\sum y_aT_{pa,b}=\sum y_aT_{a,pb}=0$ by assumption, this gives a relation involving only those $y_a$ for which $p$ divides $a$:
$$
\sum_{p|a}\ y_a (T_{a,b}-p^{-k-1}T_{pa,b})=0.
$$
Write this last expression as a sum over $\Z/n'$:
	$$\sum_{a\in \Z/n'}\ y_{pa}(T_{pa,b}-p^{-k-1}T_{p^2a,b})=0.$$
Let $T'$ be the $n'\times n'$ matrix analogous to $T$ but with $n'$ in place of $n$, so that
	$$T'_{a,b}=\L_{k+1}(\z_{n'}^{ab})=\L_{k+1}(\z_n^{pab})=T_{pa,b},$$
and write $y'_{a}=y_{pa}$ for $a\in \Z/n'$. The proof will be complete if we can show that $y'=0$.\\

The equation is now 
$$
\sum_{a\in \Z/n'}\ y'_{a}(T'_{a,b}-p^{-k-1}T'_{pa,b})=0.
$$
By inductive hypothesis the matrix $T'$ is invertible. The vector $y'$ is annihilated by the matrix $(1-p^{-k-1}E)T'$, where $E$ is the $n'\times  n'$ matrix whose non-zero entries are $E_{a,pa}=1$.  Every eigenvalue of $E$ is either zero or a root of unity (since the matrix maps the standard basis to itself). It follows that $1-p^{-k-1}E$ is also invertible. Thus $y'$ must be zero. 
\end{proof}

It remains only to prove Lemma \ref{relprime}. The proof will use multiplicative characters of $\Z/n$, in other words group homomorphisms $\chi:(\Z/n)^*\to \C^*$. In particular it will use sums of multiplicative characters against additive characters. For $m\in\Z/n$, define
	$$c(m,\chi)=\sum_{a\in (\Z/n)^*}\z_n^{ma}\bar\chi(a).$$
	
The $n\times\phi(n)$ matrix $(c(m,\chi))_{m,\chi}$ has rank $\phi(n)$. In fact, it is the product of the $n\times\phi(n)$ matrix $(\z_n^{ma})_{m\in\Z/n,a\in (\Z/n)^*}$ and the $\phi(n)\times \phi(n)$ matrix $(\bar\chi(a))_{a\in(\Z/n)^*,\chi}$. The former has rank $\phi(n)$, as it consists of some of the columns of an invertible $n\times n$ Vandermonde matrix, while the latter is invertible because characters are linearly independent.
In particular for each $\chi$ there is at least one $m\in \Z/n$ such that $c(m,\chi)\neq 0$. \\

Say that $\chi$ is \it aperiodic\rm\ if it has no period smaller than $n$, that is, if there is no proper divisor $d|n$ such that $\chi$ may be expressed as the composition $\chi'\circ\pi$ of a multiplicative character of $\Z/d$ with the canonical surjective homomorphism $\pi: (\Z/n)^*\to (\Z/d)^*$. \\

\begin{Proposition}\label{triangle}Suppose that $\chi$ is aperiodic. If $m\notin (\Z/n)^*$ then $c(m,\chi)= 0$. If $m\in (\Z/n)^*$ then $c(m,\chi)=c(1,\chi)\chi(m)\neq 0$.
\end{Proposition}

\begin{proof}[Proof] If $m\notin (\Z/n)^*$ then the order of $m$ in $\Z/n$ is some proper divisor $d$ of $n$. Let $K$ be the kernel of $\pi:(\Z/n)^*\to (\Z/d)^*$. Within any coset of $K$ in $(\Z/n)^*$ the function $a\mapsto \z_n^{am}$ is a constant. The sum of $\bar\chi$ over the coset is zero because the restriction of $\chi$ to $K$ is a nontrivial character. Thus the sum of $a\mapsto \z_n^{am}\bar\chi(a)$ is zero on each coset, and therefore the sum over the whole group is also zero.\\

If $m\in(\Z/n)^*$ then
	$$c(m,\chi)=\sum_{a\in (\Z/n)^*}\z_n^{ma}\bar\chi(a)=\sum_{a'\in (\Z/n)^*}\z_n^{a'}\bar\chi(a')\chi(m)=c(1,\chi)\chi(m),$$
by substituting $a'=am$. If one of these numbers were zero then they all would be zero. 
\end{proof}

\begin{proof}[Proof of Lemma \ref{relprime}]For any multiplicative character $\chi$ of $\Z/n$, define
	$$e_n(\chi)=\sum_{a\in (\Z/n)^*}\L_{k+1}(\z_n^a)\bar\chi(a).$$
Interpret the matrix of Lemma \ref{relprime} as a linear operator on the vector space of all functions $(\Z/n)^*\to \C$. As such it takes the basis vector $\bar\chi$ to a scalar multiple of the basis vector $\chi$, by a computation as in the proof of \ref{triangle}:
	$$\sum_a T_{m,a}\bar\chi(a)=\sum_a \L_{k+1}(\z_n^{ma})\bar\chi(a)=\sum_{a'}\L_{k+1}(\z_n^{a'})\bar\chi(a')\chi(m)=e_n(\chi)\chi(m).$$
Thus Lemma \ref{relprime} (and with it the Main Theorem) will follow from the next result.
\end{proof}

\begin{Lemma}\label{nonzero}
For any $n>0$, for any multiplicative character $\chi$ of $\Z/n$, the number $e_n(\chi)$ is different from zero.
\end{Lemma}

\begin{proof} [Proof of Lemma \ref{nonzero}]

Again the proof is by induction on $n$. In the case when $\chi$ is not aperiodic we will obtain the non-vanishing of $e_n(\chi)$ from the non-vanishing of some $e_{n'}(\chi')$ with $n'<n$. In the case when $\chi$ is aperiodic we will prove the statement directly.

Suppose that $n=pn'$ for some prime $p$, let $\pi:(\Z/n)^*\to (\Z/n')^*$ be the canonical surjective homomorphism, and suppose that $\chi=\chi'\circ \pi$ for some multiplicative character $\chi'$ of $\Z/n'$.\\

Case 1: $p$ divides $n'$. In this case we show that 
	$$e_n(\chi)=p^{-k}e_{n'}(\chi')$$
by using the Kubert identity for $\L_{k+1}$.
For any $b\in (\Z/n')^*$ the $p$-th roots of $\z_{n'}^b=\z_n^{pb}$ are $\z_n^{\tilde b}$ for $\tilde b\in\pi^{-1}(b)$. Thus we have: 
\[
	\sum_{\tilde b\in\pi^{-1}(b)} \L_{k+1}(\z_n^{\tilde b})\bar\chi(\tilde b)={p^{-k}}\L_{k+1}(\z_{n'}^b)\bar\chi'(b).
\]
The formula for $e_n(\chi)$ follows by summing over $b\in (\Z/{n'})^*$.\\

Case 2: $p$ does not divide $n'$. In this case we show that
	$$e_n(\chi)=\left({p^{-k}}-\chi(q)\right)e_{n'}(\chi'),$$
where $q\in (\Z/n')^*$ is the inverse of $p$.
For any $b\in (\Z/n')^*$ the $p$-th roots of $\z_{n'}^b=\z_n^{pb}$ are the $p-1$ elements $\z_n^{\tilde b}$ for $\tilde b\in \pi^{-1}(b)$ plus the element $\z_{n'}^{bq}$. Thus, using the Kubert identity again, we have:
	$$\sum_{\tilde b\in\pi^{-1}(b)} \L_{k+1}(\z_n^{\tilde b})\bar\chi(\tilde b)+\L_{k+1}(\z_{n'}^{bq})\bar\chi(bq)\chi(q)={p^{-k}}\L_{k+1}(\z_{n'}^b)\bar\chi'(b).$$
Again the formula follows by summing over all $b$.\\

Finally, consider the case when $\chi$ is aperiodic. We have:

	$$e_n(\chi)=\sum_{a\in (\Z/n)^*}\L_{k+1}(\z_n^a)\bar\chi(a)=\sum_{a\in (\Z/n)^*}\sum_{m=1}^\infty \frac {\z_n^{ma}}{m^{k+1}}\bar\chi(a)=\sum_{m=1}^\infty \frac {c(m,\chi)}{m^{k+1}}.$$
By Proposition \ref{triangle} this becomes
	$$c(1,\chi)\sum_{(m,n)=1}\frac {\chi(m)}{m^{k+1}}=c(1,\chi)L(\chi,k+1),$$
where $L(\chi,-)$ is a Dirichlet $L$-function. \\

By Proposition \ref{triangle} the factor $c(1,\chi)$ is different from zero. 
The value of the $L$-function at $k+1>1$ is not zero, either, because of the Euler product formula:
	$$L(\chi,s)=\prod_p\ \frac{1}{1-\chi(p)p^{-s}}.$$
\end{proof}

This concludes the proof of Theorem \ref{maintheorem}.

\bibliographystyle{plain}

\end{document}